\documentclass[12pt,a4paper]{article}
\usepackage{a4wide}
\usepackage{amsmath,amsthm,amssymb}
\usepackage{array,enumerate}
\usepackage{indentfirst}	
\usepackage{url,authblk,color}	
\usepackage{hyperref}

\newtheorem{theorem}{Theorem}[section]

\newtheorem{lemma}[theorem]{Lemma}
\newtheorem{corollary}[theorem]{Corollary}
\newtheorem{remark}[theorem]{Remark}
\newtheorem*{conj}{Conjecture}

\usepackage{tikz}
\usetikzlibrary{shapes}

\newcommand{\g}{\Gamma}

\title{A new characterization of the dual polar graphs}
\author[a]{Zhi Qiao}
\author[b]{Jack Koolen\thanks{Corresponding author}\setcounter{footnote}{-1}\footnote{E-mail addresses: zhiqiao@sicnu.edu.cn (Z. Qiao), koolen@ustc.edu.cn (J. Koolen)}}
\affil[a]{College of Mathematics and Software Science, Sichuan Normal University, 610068, Sichuan, PR China}
\affil[b]{School of Mathematical Sciences, University of Science and Technology of China, Wen-Tsun Wu Key Laboratory of the Chinese Academy of Sciences, 230026, Anhui, PR China}
\date{}

\begin{document}
\maketitle

\begin{abstract}
	In this paper we give a new characterization of the dual polar graphs, extending the work of Brouwer and Wilbrink on regular near polygons. Also as a consequence of our characterization we confirm a conjecture of the authors on non-bipartite distance-regular graphs with smallest eigenvalue at most $-k/2$, where $k$ is the valency of the distance-regular graph, in case of $c_2 \geq3$ and $a_1 =1$.
	
	{\bf Keywords}: Distance-regular graphs, Geometric distance-regular graphs, Dual polar graphs, Regular near polygons
	
	{\bf Mathematics Subject Classification}: 05C75, 05E30, 05C50
\end{abstract}

\section{Introduction}

In this paper we study geometric distance-regular graphs. 
They were introduced by Godsil \cite{G93a} as a generalization of geometric strongly regular graphs. Among the examples are the Johnson graphs, the Hamming graphs, the Grassmann graphs, the bilinear forms graphs and the dual polar graphs.
Koolen and Bang \cite{KB10} and Bang, Dubickas, Koolen and Moulton \cite{BDK15} showed that for any integer $m$ at least $2$, there are only finitely many coconnected non-geometric distance-regular graphs with smallest eigenvalue at least $-m$ and valency at least $3$. 
An important class of geometric distance-regular graphs are the regular near $2D$-gons. Brouwer and Wilbrink \cite{BW83} (completed in De Bruyn \cite{B06}) showed that the only thick regular $2D$-gons with intersection number $c_2$ at least three and $D$ at least four are the dual polar graphs (see \cite[Theorem 9.11]{DKT} and Theorem \ref{thm:dpl}). 
In this paper we extend the classification of Brouwer and Wilbrink. We give the the following characterization of the dual polar graphs
$^2A_{2D-1}(\sqrt{q})$, $B_D(q)$ and $C_D(q)$:

\begin{theorem} 	\label{thm:main1}
	Let $\Gamma$ be a non-bipartite geometric distance-regular graph with diameter $D\geq 4$ and $c_2\neq 1$. 
	Then the following hold:
	\begin{enumerate}[i)]
		\item If $a_i=c_i(a_1+1)$ $(i\leq 2)$ and $c_3=(a_1^2+a_1+1) (a_1 +c_2+1)$, then $\Gamma$ is the dual polar graph $^2A_{2D-1}(\sqrt{q})$. 
		\item If $a_i=c_i(a_1+1)$ $(i\leq 3)$ and $c_4=(a_1^2+2a_1+2)( c_3-(a_1+1)^2 )$, then $\Gamma$ is the dual polar graph $^2A_{2D-1}(\sqrt{q})$, $B_D(q)$ or $C_D(q)$.
	\end{enumerate}
\end{theorem}

The idea of the proof is an extension of the main idea used in Lang \cite{L04}. 

As a consequence of Theorem \ref{thm:main1}, we show the following:

\begin{theorem}			\label{thm:main2}
	Let $\Gamma$ be a non-bipartite distance-regular graph with valency $k\geq 3$, diameter $D\geq 2$ and smallest eigenvalue $\theta_{\min}\leq -\frac{k}{2}$. 
	If $a_1=1$ and $c_2=3$, then $\Gamma$ is one of the following graphs:
	\begin{enumerate}[i)]
		\item one of the two distance-regular graphs with intersection array $\{8,6,1;1,3,8\}$ {\em (see \cite[p.386]{BCN})},
		\item the Witt graph associated to $M_{24}$ with intersection array $\{30,28,24;1,3,15\}$ {\em (see \cite[Section 11.4]{BCN})},
		\item the dual polar graph $B_D(2)$. 
	\end{enumerate}
\end{theorem}

This result shows that the following conjecture of the authors is correct if both $c_2 \geq 3$ and $a_1=1$ hold.

\begin{conj}{\em \cite{KQ17}}
	When $D$ is large enough, a non-bipartite distance-regular graph with valency $k$, diameter $D$ and smallest eigenvalue $\theta_{\min}\leq -k/2$ is one of the following graphs:
	\begin{enumerate}[i)]
		\item the odd polygons,
		\item folded $(2D+1)$-cubes,
		\item the odd graphs $O_k$,
		\item the Hamming graphs $H(D,3)$,
		\item the dual polar graphs $^2A_{2D-1}(2)$,
		\item the dual polar graphs $B_D(2)$.
	\end{enumerate}
\end{conj}

Theorem \ref{thm:main2} improves earlier results of the authors in \cite{KQ1702} in which we showed that such a graph is the dual polar graph $^2A_{2D-1}$ when $a_1=1$ and $c_2\geq 4$. 
In this paper, we first give a characterization of dual polar graphs.  
Then we show the conjecture is true with $a_1=1$ and $c_2=3$, where if $D\geq 4$, such a graph is $B_D(2)$. 

This paper is organized as follows. 
In the next section, we give the definitions and some preliminary results. 
Then in Section \ref{sec:char}, we give a characterization of dual polar graphs and prove Theorem \ref{thm:main1}. 
In the last section we give a proof of Theorem \ref{thm:main2}. 

\section{Preliminaries}
For more background, see \cite{BCN} and \cite{DKT}. 

All the graphs considered in this paper are finite, undirected and simple. 
Let $\Gamma$ be a graph with vertex set $V=V(\Gamma)$ and edge set $E=E(\Gamma)$. 
Denote $x\sim y$ if the vertices $x,y\in V$ are adjacent. 
The {\em distance} $d(x,y)=d_\Gamma(x,y)$ between two vertices $x,y\in V(\g)$ is the length of a shortest path connecting $x$ and $y$. 
The maximum distance between two vertices in $\Gamma$ is the {\em diameter} $D=D(\Gamma)$. 
We use $\Gamma_i(x)$ for the set of vertices at distance $i$ from $x$ and write, for the sake of simplicity, $\Gamma(x):=\Gamma_1(x)$. 
For two vertices $x,y\in V$, we denote $\Gamma^i_j(x,y):=\Gamma_i(x)\cap\Gamma_j(y)$. 
The {\em valency} of $x$ is the number $|\Gamma(x)|$ of vertices adjacent to it. 
A graph is {\em regular with valency $k$} if the valency of each of its vertices is $k$.
A graph Γ is called {\em bipartite} if it has no odd cycle.

A connected graph $\Gamma$ with diameter $D$ is called {\em distance-regular} if there are integers $b_i$, $c_i$ ($0\leq i \leq D$) such that for any two vertices $x,y\in V(\Gamma)$ with $d(x,y)=i$, there are exactly $c_i$ neighbors of $y$ in $\Gamma_{i-1}(x)$ and $b_i$ neighbors of $y$ in $\Gamma_{i+1}(x)$, where we define $b_D=c_0=0$. 
In particular, $\Gamma$ is a regular graph with valency $k:=b_0$. 
We define $a_i:=k-b_i-c_i$ $(0 \leq i\leq D)$ for notational convenience. Note that $a_i =|\Gamma(y)\cap\Gamma_i(x)|$ holds for any two vertices $x,y$ with $d(x,y)=i$ $(0\leq i\leq D)$.

For a distance-regular graph $\Gamma$ and a vertex $x\in V(\Gamma)$, we denote $k_i:=|\Gamma_i(x)|$ and $p_{ij}^h:=|\{w\mid w\in \g_i(x)\cap\g_j(y)\}|$ for any $y\in\g_h(x)$. 
It is easy to see that $k_i = b_0b_1\cdots b_{i-1}/(c_1c_2\cdots c_i)$ and hence it does not depend on $x$. 
The numbers $a_i$ , $b_i$ and $c_i$ ($0\leq i\leq D$) are called the {\em intersection numbers}, and the array $\{b_0,b_1,\ldots,b_{D-1};c_1,c_2,\ldots,c_D\}$ is called the {\em intersection array} of $\g$. 

Let $\g$ be a distance-regular graph with $v$ vertices and diameter $D$. 
Let $A_i$ $(0\leq i\leq D)$ be the $(0,1)$-matrix whose rows and columns are indexed by the vertices of $\g$ and the $(x, y)$-entry is $1$ whenever $d(x,y)=i$ and $0$ otherwise. 
We call $A_i$ the {\em distance-$i$ matrix} and $A:=A_1$ the {\em adjacency matrix} of $\g$. 
The {\em eigenvalues} $\theta_0>\theta_1>\cdots>\theta_D$ of the graph $\Gamma$ are just the eigenvalues of its adjacency matrix $A$. 
We denote $m_i$ the {\em multiplicity} of $\theta_i$. 

For each eigenvalue $\theta_i$ of $\Gamma$, let $U_i$ be a matrix with its columns forming an orthonormal basis for the eigenspace associated with $\theta_i$. 
And $E_i:=U_i U_i^T$ is called the {\em minimal idempotent} associated with $\theta_i$, satisfying $E_i E_j= \delta_{ij} E_j$ and $A E_i=\theta_i E_i$, where $\delta_{ij}$ is the Kronecker delta. 
Note that $vE_0$ is the all-ones matrix $J$. 

The set of distance matrices $\{A_0=I,A_1,A_2,\ldots,A_D\}$ forms a basis of a commutative $\mathbb R$-algebra $\mathcal A$, known as the {\em Bose-Mesner algebra}. 
The set of minimal idempotents $\{E_0=\frac{1}{v}J,E_1,E_2,\ldots,E_D\}$ is another basis of $\mathcal A$. 
There exist $(D+1)\times(D+1)$ matrices $P$ and $Q$ (see \cite[p.45]{BCN}), such that the following relations hold 
\begin{equation}	\label{eq:pq}
A_i=\sum_{j=0}^D P_{ji}E_j \quad \text{and} \quad E_i=\frac{1}{v}\sum_{j=0}^D Q_{ji}A_j\quad (0\leq i\leq D).
\end{equation}
Note that  $Q_{0i}=m_i$ (see \cite[Lemma 2.2.1]{BCN}). 

Let $E_i=U_i U_i^T$ be the minimal idempotent associated with $\theta_i$, where the columns of $U_i$ form an orthonormal basis of the eigenspace associated with $\theta_i$. 
We denote the $x$-th row of $\sqrt{v/m_i}U_i$ by $\hat{x}$. 
Note that $E_i\circ A_j=\frac{1}{v}Q_{ji}A_j$, hence all the vectors $\hat{x}$ are unit vectors and the cosine of the angle between two vectors $\hat{x}$ and $\hat{y}$ is $u_j(\theta_i):=\frac{Q_{ji}}{Q_{0i}}$, where $d(x,y)=j$.
The map $x\mapsto \hat{x}$ is called a {\em normalized representation} and the sequence $(u_j(\theta_i))_{j=0}^D$ is called the {\em standard sequence} of $\Gamma$, associated with $\theta_i$. 
As $A U_i=\theta_i U_i$, we have $\theta_i \hat{x}=\sum_{y\sim x} \hat{y}$, and hence the following holds:
\begin{equation}\label{eq:std}
\left\{	
	\begin{aligned}
		& c_j u_{j-1}(\theta_i)+ a_j u_j(\theta_i)+ b_j u_{j+1}(\theta_i)=\theta_i u_j(\theta_i)\quad(1\leq j\leq D-1), \\
		& c_Du_{D-1}(\theta_i)+a_D u_{D}(\theta_i)=\theta_i u_D(\theta_i),
	\end{aligned}
\right.
\end{equation}
with $u_0(\theta_i)=1$ and $u_1(\theta_i)=\frac{\theta_i}{k}$. 

\begin{lemma}{\em \cite[Theorem 2.8]{DKT}}\label{lm:biggs}
	Let $\Gamma$ be a distance-regular graph with diameter $D$ and $v$ vertices. Let $\theta$ be an eigenvalue of $\Gamma$ and $(u_i)_{i=0}^D$ be the standard sequence associated with $\theta$. Then the multiplicity $m(\theta)$ of $\theta$ as an eigenvalue of $\Gamma$ satisfies
    \begin{align}\label{eq:biggs}
	   m(\theta)=\frac{v}{\sum_{i=0}^D k_i u_i^2}. 
    \end{align}
\end{lemma}

\begin{lemma} \label{lm:f}
	{\em \cite[Proposition 4.1.6]{BCN}}
	Let $\Gamma$ be a distance-regular graph with valency $k$ and diameter $D$. Then the following conditions hold:
	\begin{enumerate}[i)]
		\item $1=c_1\leq c_2\leq \cdots \leq c_D$,
		\item $k=b_0\geq b_1\geq \cdots \geq b_{D-1}$, 
		\item $k_i$'s ($1\leq i\leq D$) are positive integers,
		\item the multiplicities are positive integers.
	\end{enumerate}
\end{lemma}

Let $\Pi=\{P_1, \ldots,P_t\}$ be a partition of the vertex set of a graph $\Gamma$. Let $f_{ij}$ $(1\leq i,j\leq t)$ be the average number of neighbors in $P_j$ of a vertex in $P_i$. The partition $\Pi$ is called {\em equitable} if for all $1\leq i,j\leq t$, every vertex in $P_i$ has exactly $f_{ij}$ neighbors in $P_j$.	
	
The following result was first shown by Delsarte \cite{D73} for strongly regular graphs, and extended by Godsil \cite{G93} to distance-regular graphs.

\begin{lemma}
	Let $\Gamma$ be a distance-regular graph with valency $k$ and smallest eigenvalue $\theta_{\min}$. Let $C$ be a clique in $\Gamma$ with $c$ vertices. Then $c\leq 1-\frac{k}{\theta_{\min}}$. 
\end{lemma}

A clique $C$ in a distance-regular graph $\Gamma$ that attains this Delsarte bound is called a {\em Delsarte clique}. 
A distance-regular graph Γ is called {\em geometric (with respect to $\mathcal C$)} if it contains a collection $\mathcal C$ of Delsarte cliques such that each edge is contained in a unique $C\in \mathcal C$. 

Let $\Gamma$ be a connected graph with diameter $D$, and let $C$ be a subset of $V(\g)$. For $i\geq 0$, let $C_i=\{x\in V\mid d(x,C)=i\}$, where $d(x,C)=\min\{d(x,c)\mid c\in C\}$. The {\em covering radius} of $C$, denoted by $\rho=\rho(C)$, is the maximum $i$ such that $C_i\neq \emptyset$. The subset (or code) C is called {\em completely regular} if the distance partition $\Pi=\{C_i\mid i=0,1,\ldots,\rho\}$ is equitable. 

\begin{lemma}	\label{lm:ds}
	{\em \cite[Proposition 4.2]{DKT}}
	Let $\Gamma$ be a geometric distance-regular graph with diameter $D$ and smallest eigenvalue $\theta_{\min}$. Then any $(a_1+2)$-clique $C$ in $\Gamma$ is a completely regular code with covering radius $D-1$ and $\gamma_iu_i+(a_1+2-\gamma_i)u_{i+1}=0$ $(0\leq i\leq D-1)$, where $(u_i)_{i=0}^D$ is the standard sequence associated with $\theta_{\min}$ and $\gamma_i=|\Gamma_i(x)\cap C|$ for a vertex $x$ at distance $i$ from $C$.  
\end{lemma}	

Let $\Gamma$ be a geometric distance-regular graph with diameter $D$, then the parameter $\gamma_i:=|\Gamma_i(x)\cap C|$ $(0\leq i\leq D-1)$ is well-defined, where $C$ is any Delsarte clique and $x$ is any vertex at distance $i$ from $C$. 
\begin{lemma} \label{lm:gi}
	Let $\Gamma$ be a geometric distance-regular graph with diameter $D$. Then $1\leq \gamma_i\leq a_1+1$ and $\gamma_i$ is increasing ($0\leq i\leq D-1$). 
\end{lemma}
\begin{proof}
	Let $\theta=-\frac{k}{a_1+1}$ be the smallest eigenvalue of $\g$ with associated standard sequence $(u_i)_{i=0}^D$. 
	By definition, we see that $1\leq \gamma_i\leq a_1+2$ $(0\leq i\leq D-1)$. 
	Note $u_0=1=\gamma_0$, by Lemma \ref{lm:ds}, it follows that $(a_1+2-\gamma_i)u_{i+1}\neq 0$, that is $u_{i+1}\neq 0$ and $\gamma_i\leq a_1+1$ $(0\leq i\leq D-1)$. 
	For any $0\leq i\leq D-2$, we can take two vertex $x,y$ with $d(x,C)=i+1$, $d(x,y)=1$ and $d(y,C)=i$. 
	Then we see that $\g_i(y)\cap C \subseteq \Gamma_{i+1}(x)\cap C$, that is $\gamma_i\leq \gamma_{i+1}$. 
\end{proof}

Let $\Gamma$ be a geometric distance-regular graph with diameter $D\geq 2$, and let $\theta_{\min}$ be the smallest eigenvalue of $\Gamma$ with associated standard sequence $(u_i)_{i=0}^D$ and normalized representation $x\mapsto \hat{x}$.  
For any $2\leq j\leq D$, choose two vertices $x,y$ with $d(x,y)=j$ and define $F_j$,  $C_j$ and $S_j^C$ with respect to $x,y$ as the following

\begin{figure}[h!]
	\begin{center}
		\begin{tikzpicture}
		\node (x) 	at (0,0) 		[circle, draw] 	{$x$};
		\node (zc) 	at (2.5,0) 		[ellipse, draw] {$\Gamma^{1}_{j-1}(x,y)$};
		\node (wc) 	at (8,0) 		[ellipse, draw]	{$\Gamma^{j-1}_{1}(x,y)$};
		\node (y) 	at (10.5,0) 	[circle, draw] 	{$y$};
		
		\path
		(x) 	edge 			(zc)
		(zc) 	edge[dashed] 	(wc)
		(wc) 	edge 			(y);
		\end{tikzpicture}
	\end{center}
	\caption{Diagram of $F_j$ and $C_j$}
\end{figure}
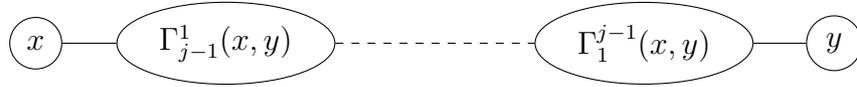

\begin{align}
	F_j   & =\hat{x}-\hat{y},                                                                            \\
	C_j   & =\sum_{z\in \Gamma^{1}_{j-1}(x,y)} \hat{z}-\sum_{w\in \Gamma^{j-1}_{1}(x,y)}\hat{w},         \\
	S_j^C & =\langle C_j,C_j\rangle \langle F_j,F_j\rangle-\langle C_j,F_j\rangle\langle C_j,F_j\rangle.
\end{align} 
By definition, we have 
\begin{align}
	\langle F_j,F_j\rangle & =2(u_0-u_j),       \label{eq:ff} \\
	\langle C_j,F_j\rangle & =2c_j(u_1-u_{j-1}).	\label{eq:cf}
\end{align}
Moreover, if $\gamma_{j-1}=1$, then 
\begin{align}
	\langle C_j,C_j\rangle =2c_j((u_0+(c_j-1)u_2)-(c_{j-1}u_{j-2}+(c_j-c_{j-1})u_j)),	\label{eq:cc}
\end{align}
and $S^C_j$ is independent of the choice of the representation. 
Note that $S^C_j\geq 0$ and equality holds if and only if $F_j$ and $C_j$ are linearly dependent. 

A distance-regular graph $\Gamma$ is {\em of order $(s, t)$} (for some integers $s$ and $t$) if it is locally the disjoint union of $t + 1$ cliques of size $s$. 
A distance-regular graph $\Gamma$ of order $(s,t)$ with diameter $D$ is called a {\em regular near $2D$-gon} if $a_i = c_ia_1$ $(1\leq i\leq D$), and such a graph is geometric. We call $\g$ {\em thick} if $s\geq 2$. 

\begin{theorem}	\label{thm:dpl}
	{\em (c.f. \cite[Theorem 6.6.1]{BCN}, \cite[Theorem 9.11]{DKT})}
	Let $\Gamma$ be a thick regular near $2D$-gon with diameter $D\geq 4$ and $c_2\neq 1$. Then $\Gamma$ is a Hamming graph or a dual polar graph. 
\end{theorem}

\begin{lemma} \label{lm:2dgon}
	Let $\Gamma$ be a geometric distance-regular graph with diameter $D\geq 2$. Then the following holds:
	\begin{equation}	\label{eq:2dgon}
		\left\{
			\begin{aligned}
				a_i & =c_i \frac{a_1+1-\gamma_{i-1}}{\gamma_{i-1}} + b_i \frac{\gamma_i-1}{a_1+1-(\gamma_{i}-1)}\quad (1\leq i\leq D-1), \\
				a_D & =c_D \frac{a_1+1-\gamma_{D-1}}{\gamma_{D-1}}. 
			\end{aligned}
		\right.
	\end{equation}
	Moreover, $\Gamma$ is a regular near $2D$-gon if and only if $\gamma_{D-1}=1$. 
\end{lemma}
\begin{proof}
	Let $(u_i)_{i=0}^D$ be the standard sequence associated with the smallest eigenvalue $\theta_{\min}=-\frac{k}{a_1+1}$. 
	Then we see that $1\leq \gamma_i\leq a_1+1$ ($0\leq i\leq D-1$) by Lemma \ref{lm:gi}, and  Lemma \ref{lm:ds} implies 
	\begin{equation}\label{eq:ui}
	\left\{
		\begin{aligned}
			u_{i+1} & =	-\frac{\gamma_{i}}{a_1+2-\gamma_{i}}u_{i},                             \\
			u_{i-1} & =	-\frac{a_1+2-\gamma_{i-1}}{\gamma_{i-1}}u_{i} \quad (1\leq i\leq D-1).
		\end{aligned}
	\right.
	\end{equation}
	Then substitute $\theta=-\frac{c_i+a_i+b_i}{a_1+1}$ and Equation (\ref{eq:ui}) into Equation (\ref{eq:std}), we obtain Equation (\ref{eq:2dgon}). 
	And we see that $\gamma_{D-1}=1$ if and only if $a_i=c_i a_1$ $(0\leq i\leq D)$, that is $\g$ is a regular near $2D$-gon. 
\end{proof}

A subgraph $\Delta$ of $\g$ is called {\em strongly closed} if $z\in V(\Delta)$ for all vertices $x, y \in V(\Delta)$ and $z\in V(\Gamma)$, such that $d_\Gamma(x,z)+d_\Gamma(z,y)\leq d_\Gamma(x,y)+1$. 
A distance-regular graph $\Gamma$ with diameter $D$ is said to be {\em $m$-bounded} for some $m =1,2,\ldots , D$ if for all $i=1,2,\ldots, m$ and all vertices $x$ and $y$ at distance $i$ there exists a strongly closed subgraph $\Delta(x,y)$ with diameter $i$, containing $x$ and $y$ as vertices. 

\begin{lemma}{\em (c.f. \cite[Theorem 1.1]{H99})}	\label{lm:mbd}
	Let $\Gamma$ be a non-bipartite geometric distance-regular graph with diameter $D\geq 3$ and $c_2\neq 1$. Then the following are equivalent:
	\begin{enumerate}[i)]
		\item $\gamma_{m}=1$,
		\item $\Gamma$ is $m$-bounded. 
	\end{enumerate}
\end{lemma}
\begin{proof}
	Assume $\gamma_{m}=1$. Then $\gamma_i=1$ ($0\leq i\leq m$) and Lemma \ref{lm:2dgon} implies $a_i=c_ia_1$ ($1\leq i\leq m$). 
	Since $c_2\neq 1$, by \cite[Theorem 5.2.1]{BCN}, we see $c_i\geq c_{i-1}+1$ $(1\leq i\leq m)$. 
	Then \cite[Proposition 11.3]{DKT} implies that $\Gamma$ is $m$-bounded. 
	
	Now we assume $\g$ is $m$-bounded. Then \cite[Proposition 11.3]{DKT} implies $a_i=c_ia_1$ ($1\leq i\leq m$). 
	It follows from Lemma \ref{lm:2dgon} that $\gamma_i=1$ ($0\leq i\leq m$).  
\end{proof}

\section{Characterization of dual polar graphs}	\label{sec:char}

In this section we will characterize the dual polar graphs. 

\begin{lemma}
	Let $\Gamma$ be a non-bipartite geometric distance-regular graph with diameter $D\geq 4$, and let $\theta\neq k$ be an eigenvalue of $\Gamma$ with associated standard sequence $(u_i)_{i=0}^D$. If there exists some $j$ satisfying $3\leq j\leq D-1$, such that $S_j^C=0$, then the following hold: 
	\begin{align}\label{eq:rec}
	(u_i-u_{i-j+2})=\frac{u_1-u_{j-1}}{u_0-u_j}(u_{i+1}-u_{i-j+1})\quad (j\leq i\leq D-1).
	\end{align}
\end{lemma}

\begin{proof}	
	Take three vertices $x,y,u$ with  $d(x,u)=i+1$, $d(x,y)=j$ and $d(y,u)=i-j+1$ $(j\leq i\leq D-1)$. Consider a normalized representation $w\mapsto \hat{w}$ associated with $\theta$ and $C_j,F_j$ with respect to $x,y$. 
	Then 
	\begin{align}
	\langle \hat{u}, F_j\rangle & =(u_{i+1}-u_{i-j+1}),	\label{eq:uf}\\
	\langle \hat{u}, C_j\rangle & =c_j(u_i-u_{i-j+2}).  \label{eq:uc}
	\end{align} 
	
	\begin{figure}[h!]
		\begin{center}
			\begin{tikzpicture}
			\node (x) 	at (0,0) 		[circle, draw] 	{$x$};
			\node (zc) 	at (2.5,0) 		[ellipse, draw] {$\Gamma^{1}_{j-1}(x,y)$};
			\node (wc) 	at (7,0) 		[ellipse, draw]	{$\Gamma^{j-1}_{1}(x,y)$};
			\node (y) 	at (9.5,0) 	[circle, draw] 	{$y$};
			\node (v)  	at (11,0) 		[circle, draw] 	{$u$};
			
			\path
			(x) 	edge 			(zc)
			(zc) 	edge[dashed] 	(wc)
			(wc) 	edge 			(y)
			(y)		edge[dashed]	(v);
			\end{tikzpicture}
		\end{center}
	\end{figure}
	
	By \cite[Proposition 4.4.7]{BCN}, we see that $u_j \neq u_0$ for $1 \leq j \leq D-1$ as the graph is not bipartite. It follows from Equation (\ref{eq:ff}) that $F_j$ is a non-zero vector. 
	Then we see $S^C_j=0$ implies $C_j$ and $F_j$ are linearly dependent, that is, there exists a constant $t_j$ such that $C_j=t_j F_j$. 
	Substitute Equation (\ref{eq:ff}) and (\ref{eq:cf}) into $t_j=\frac{\langle C_j,F_j\rangle}{\langle F_j,F_j\rangle}$, we see that 
	\begin{equation}	\label{eq:tj}
		t_j=c_j\frac{u_1-u_{j-1}}{u_0-u_j}. 
	\end{equation}
	Substitute Equation (\ref{eq:uf}) (\ref{eq:uc}) and (\ref{eq:tj}) into $\langle \hat{u}, C_j\rangle=t_j \langle \hat{u}, F_j\rangle$, the result follows. 
\end{proof}

\begin{remark}	\label{rmk:a}
	\begin{enumerate}[i)]
		\item When $j=2$, then $C_j$ is the zero vector and Equation (\ref{eq:rec}) always holds. 
		\item The standard sequence $(u_i)_{i=1}^D$ satisfies the following relations, and thus it is determined by the numbers $u_0, u_1, \ldots, u_j$. 
		$$
		\left\{
			\begin{aligned}
				u_{i+1} & =(u_i-u_{i-j+2})\frac{u_0-u_j}{u_1-u_{j-1}}+u_{i-j+1}, & \text{if $u_1\neq u_{j-1}$}, \\
				u_{i+1} & =u_{i-j+3},                                            & \text{if $u_1= u_{j-1}$}.
			\end{aligned}
		\right. 
		$$
	\end{enumerate}
\end{remark}

\begin{lemma} {\em (c.f. \cite[Theorem 6.3]{BV13})}
	\label{lm:relation}
	Let $\Gamma$ be a non-bipartite geometric distance-regular graph with diameter $D\geq 4$ and $\gamma_{j-1}=1$. 
	Then the following hold:
	\begin{enumerate}[i)]
		\item If $j=3$, then $c_3\leq (a_1^2+a_1+1) (a_1 +c_2+1)$ and equality holds if and only if $S^C_3=0$. 
		\item If $j=4$, then $c_4\geq (a_1^2+2a_1+2) ( c_3-(a_1+1)^2 )$ and equality holds if and only if $S^C_4=0$.
	\end{enumerate}
\end{lemma}
\begin{proof}
	Let $(u_i)_{i=0}^D$ be the standard sequence associated with the smallest eigenvalue $\theta_{\min}=-\frac{k}{a_1+1}$, 
	Since $\gamma_{j-1}=1$, Lemma \ref{lm:ds} implies that $u_i=(-\frac{1}{a_1+1})^i$ $(0\leq i\leq j)$. 
	
	By Equation (\ref{eq:ff}) (\ref{eq:cf}) and (\ref{eq:cc}), when $j=3$, we see that 
	\begin{align}	
		S^C_3=\frac{4(a_1+2)^2a_1c_3}{(a_1+1)^6} ((a_1^2+a_1+1) (a_1 +c_2+1)-c_3).	\label{eq:s3}
	\end{align}
	Then $i)$ follows. 
	
	Similarly, when $j=4$, we have
	\begin{align}	
		S^C_4=\frac{4(a_1+2)^2a_1^2c_4}{(a_1+1)^8} (c_4-(a_1^2+2a_1+2)( c_3-(a_1+1)^2 ) ).\label{eq:s4}
	\end{align}
	And $ii)$ follows. 
\end{proof}

\begin{lemma}\label{lm:sjc}
	Let $\Gamma$ be a non-bipartite geometric distance-regular graph with diameter $D\geq 4$. If there exists some $j$ satisfying $3\leq j\leq D-1$, such that $\gamma_{j-1}=1$ and $S_j^C=0$, then $\Gamma$ is a regular near $2D$-gon. 
\end{lemma}	
\begin{proof}
	Let $(u_i)_{i=0}^D$ be the standard sequence associated with the smallest eigenvalue $\theta_{\min}=-\frac{k}{a_1+1}$.
	Since $\gamma_{j-1}=1$, by Lemma \ref{lm:ds}, we see $u_i=(-\frac{1}{a_1+1})^i$ ($0\leq i\leq j$). 
	Remark \ref{rmk:a} implies that $u_i=(-\frac{1}{a_1+1})^i$ ($0\leq i\leq D$). 
	Then we see $\gamma_i=1$ ($0\leq i\leq D-1$) by Lemma \ref{lm:ds}. 
	It follows from Lemma \ref{lm:2dgon} that $\Gamma$ is a regular near $2D$-gon. 
\end{proof}

Now we give a proof of Theorem \ref{thm:main1}. 


\begin{proof}[Proof of Theorem \em \ref{thm:main1}. ]
	\begin{enumerate}[\it i)]
		\item 
		Since $a_i=c_i(a_1+1)$ $(i\leq 2)$, we see that $\gamma_2=1$ by Lemma \ref{lm:2dgon}.
		As $c_3=(a_1^2+a_1+1) (a_1 +c_2+1)$, we have $S_3^C=0$ by Lemma \ref{lm:relation}, and $\g$ is a regular near $2D$-gon by Lemma \ref{lm:sjc}. 
		Since $c_2\neq 1$, we see that $\g$ is a Hamming graph or a dual polar graph by Theorem \ref{thm:dpl}. 
		Then we check the intersection arrays (see \cite[Section 9.2 and 9.4]{BCN}) and the result follows. 
		\item The result follows in the same way.  
	\end{enumerate}
\end{proof}

\begin{corollary}	\label{cor:1}
	Let $\Gamma$ be a non-bipartite geometric distance-regular graph with diameter $D\geq 4$ and $c_2\neq 1$. 
	Then the following hold:
	\begin{enumerate}[i)]
		\item If $\gamma_2=1$ and $c_2\geq (a_1+1)^2+1$, then $\Gamma$ is the dual polar graph $^2A_{2D-1}(\sqrt{q})$. 
		\item If $\gamma_3=1$, $c_2\geq a_1+2$ and $c_4\leq (a_1+2)(a_1^2+2a_1+2)$, then $\Gamma$ is the dual polar graph $B_D(q)$ or $C_D(q)$. 
	\end{enumerate}
\end{corollary}
\begin{proof}
	\begin{enumerate}[\it i)]
		\item $\gamma_2=1$ implies that $\Gamma$ is $2$-bounded by Lemma \ref{lm:mbd}. 
		Then $ c_2^2-c_2+1\leq c_3 \leq (a_1^2+a_1+1) (a_1 +c_2+1)$, where the inequalities are followed from \cite[Proposition 19]{H09} and Lemma \ref{lm:relation}, respectively. 
		We see $c_2\geq (a_1+1)^2+1$ implies that $c_3 = (a_1^2+a_1+1) (a_1 +c_2+1)$ and the result follows from Theorem \ref{thm:main1}. 
		
		\item $\gamma_3=1$ implies that $\Gamma$ is $3$-bounded by Lemma \ref{lm:mbd}. 
		Then $c_3\geq c_2^2-c_2+1$ by \cite[Proposition 19]{H09}. 
		Together with Lemma \ref{lm:relation}, we see $c_4\geq (a_1^2+2a_2+2)(c_2^2-c_2+1-(a_1+1)^2)$. 
		And $c_2\geq a_1+2$ and $c_4\leq (a_1+2)(a_1^2+2a_1+2)$ implies that $c_4= (a_1^2+2a_1+2) ( c_3-(a_1+1)^2 )$ and $c_2=a_1+2$. 
		The result follows from Theorem \ref{thm:main1}. 
	\end{enumerate}
\end{proof}

\section{Proof of Theorem \ref{thm:main2}}
\begin{proof}[Proof of Theorem \em \ref{thm:main2}]
	As $a_1=1$ and the smallest eigenvalue $\theta_{\min}=-\frac{k}{2}$, we see that each triangle is a Delsarte clique, and $\g$ is geometric. 
	
	If the diameter $D\leq 4$, the result follows from \cite[Theorem 1.2]{KQ17}. So we may assume $D\geq 5$. 
	
	Note that $1\leq \gamma_i\leq a_1+1=2$ and $\gamma_i$ is increasing ($0\leq i\leq D-1$) by Lemma \ref{lm:gi}. If $\gamma_{D-1}=1$, then $\Gamma$ is a regular near $2D$-gon by Lemma \ref{lm:2dgon}. 
	Then by \cite[Theorem 9.11]{DKT}, we see that $\Gamma$ is $B_D(2)$. 
	So we may assume $j$ is the smallest number satisfying $2\leq j\leq D-1$, such that $\gamma_j=2$. 
	If $j\geq 5$, then $\gamma_4=1$. 
	And $\Gamma$ is $4$-bounded by Lemma \ref{lm:mbd}, with $\Delta_4$ a strongly closed subgraph with diameter $4$. 
	Then $\Delta_4$ is a regular near octagon with the same intersection number $c_i$ and $a_i$ $(i\leq 4)$ as $\Gamma$ (especially $c_2=3$). 
	By \cite[Theorem 9.11]{DKT}, we see that $\Gamma$ is $B_4(2)$. 
	Then by Theorem \ref{thm:main1}, we see that $\Gamma$ is the dual polar graph $B_D(q)$. 
	So we may assume $j\leq 4$. 
	Let $(u_i)_{i=0}^D$ be the standard sequence associated with $\theta_{\min}$. Then by Lemma \ref{lm:ds}, we see that 
	\begin{align}
	\left\{
	\begin{aligned}
	&u_i  =(-\frac{1}{2})^i,&\quad i&\leq j,      \\
	&u_i  =(-\frac{1}{2})^{2j-i},&\quad i&\geq j.	\label{eq:ud}
	\end{aligned}
	\right.
	\end{align}
	Since $|u_i|\leq 1$, we see that $D\leq 2j$. Since $j\leq 4$, we see $D\leq 8$. 
	
	By Lemma \ref{lm:biggs}, we see that the multiplicity $m$ of $\theta_{\min}$ satisfies 
	\begin{align}
		m=\frac{\sum_{i=0}^D k_i}{\sum_{i=0}^D k_i u_i^2}\leq \max_{0\leq i\leq D} \frac{1}{u_i^2}=4^j.
	\end{align} 
	By \cite[Theorem 4.4.4]{BCN}, we have $k\leq 2m \leq 2^{2j+1}$. 
	
	Let $\Delta$ be a strongly closed subgraph of $\g$ with diameter $j-1$. 
	Note that $\Delta$ is determined by any pair of vertices in $V(\Delta)$ at distance $j-1$. 
	Choose two vertices $x,y\in V(\Delta)$ with $d(x,y)=j-2$. 
	Denote $S=\g(y)\cap\g_{j-1}(x)$ with $|S|=b_{j-2}$ in the graph $\g$, 
	and $T=S\cap V(\Delta)$ with $|T|=b_{j-2}'$, where $b_i'$ denotes the corresponding intersection numbers of $\Delta$. 
	Then for any $z\in S\cap V(\Delta)$, the strongly closed subgraph with diameter $j-1$ containing $x,z$ is also $\Delta$. 
	It follows that $S$ is partitioned by its intersection with strongly closed subgraphs containing $x$ with diameter $j-1$ , and 
	\begin{equation}	\label{eq:divk}
		b_{j-2}'\mid b_{j-2}=k-2c_{j-2}.
	\end{equation}  
	
	Now we consider $5\leq D\leq 8$. As $\frac{D}{2}\leq j\leq 4$, the only possible $(j,D)$ are the following
	$$
	L=\{(3,5),(4,5),(3,6),(4,6),(4,7),(4,8)\}.
	$$
	By Lemma \ref{lm:2dgon}, we see that 
	\begin{align}
	\left\{
	\begin{aligned}	\label{eq:int}
	& a_i  =  c_i, & \quad i & \leq j-1, \\
	& a_i  =  k/2, & \quad i & =j,       \\
	& a_i =  b_i,  & \quad i & \geq j+1.
	\end{aligned}
	\right.
	\end{align}
	
	If $j=D/2$, by Equation (\ref{eq:ud}), we see that $u_D=1$ and $\Gamma$ is antipodal by \cite[Proposition 4.4.7]{BCN}. Then $b_i=c_{D-i}$ for $i\neq D/2$ by \cite[Proposition 4.2.2]{BCN}. 
	When $(j,D)$ is $(3,6)$ or $(4,8)$, by Equation (\ref{eq:int}), the intersection number of $\Gamma$ is determined by $c_3$ or $(c_3,c_4)$, respectively. 
	If $j\neq D/2$, then the intersection number of $\Gamma$ is determined by $(c_3,\ldots,c_{D-1})$. 
	
	If $j=4$, as $\Gamma$ is $3$-bounded, $\Gamma$ contains a strongly closed subgraph $\Delta_3$ with diameter $3$, which is a regular near hexagon. 
	Then by \cite[Theorem 1.2]{KQ17}, the graph $\Delta_3$ is $B_D(2)$ or the Witt graph associated to $M_{24}$, with $c_3=7$ or $15$, respectively.
	
	Now we give a list of necessary conditions:
	\begin{enumerate}[\it i)]
		\item $k\leq 2^{2j+1}$.
		\item If $(j,D)=(3,6)$, then $4\mid k-2$ (Equation (\ref{eq:divk})) and the array is determined by $c_3$.
		\item If $(j,D)=(4,8)$, then $c_3=7$ and $8\mid k-6$, or $c_3=15$ and $24\mid k-6$. The array is determined by $(c_3,c_4)$.
		\item If $(j,D)=(3,5)$, then $4\mid k-2$ and the array is determined by $(c_3,c_4)$.
		\item If $(j,D)=(4,5),(4,6)$ or $(4,7)$, then $c_3=7$ and $8\mid k-6$, or $c_3=15$ and $24\mid k-6$. The array is determined by $(c_3,\ldots,c_{D-1})$.
	\end{enumerate}
	
	For all possible pairs $(j,D)\in L$, no array $(k,b_1,\ldots, b_D;c_1,c_2,c_3,\ldots,c_D)$ satisfies the above necessary conditions and Lemma \ref{lm:f}. And the result follows.     
\end{proof}
\begin{remark}
	\begin{enumerate}[i)]
		\item The dual polar graphs $B_D(2)$ and $C_D(2)$ are isomorphic. 
		\item See \cite[Theorem 1.1]{KQ1702} for a result similar to Corollary \ref{cor:1} i). And for the case $c_2=5$ in Theorem \ref{thm:main2}, see \cite[Theorem 1.2]{KQ1702}.	
	\end{enumerate}
\end{remark}
\noindent{\bf Acknowledgments.}\\
Z.Q. is partially supported by Sichuan Normal University (Grant No. 341569001). J.H.K. is partially supported by the National Natural Science Foundation of China (Grant No. 11471009 and Grant No. 11671376).

\bibliographystyle{plain}
\bibliography{bib}

\begin{thebibliography}{10}

\bibitem{BDK15}
S.~Bang, A.~Dubickas, J.~Koolen, and V.~Moulton.
\newblock There are only finitely many distance-regular graphs of fixed valency
  greater than two.
\newblock {\em Adv. Math.}, 269:1--55, 2015.

\bibitem{BCN}
A.~Brouwer, A.~Cohen, and A.~Neumaier.
\newblock {\em Distance-Regular Graphs}.
\newblock Springer-Verlag, Berlin, 1989.

\bibitem{BW83}
A.~Brouwer and H.~Wilbrink.
\newblock The structure of near polygons with quads.
\newblock {\em Geom. Dedicata}, 14:145--176, 1983.

\bibitem{B06}
B.~De Bruyn.
\newblock The completion of the classification of the regular near octagons
  with thick quads.
\newblock {\em J. Algebraic Combin.}, 24:23--29, 2006.

\bibitem{BV13}
B.~De Bruyn and F.~Vanhove.
\newblock Inequalities for regular near polygons, with applications to
  $m$-ovoids.
\newblock {\em European J. Combin.}, 34:522--538, 2013.

\bibitem{D73}
P.~Delsarte.
\newblock {\em An algebraic approach to the association schemes of coding
  theory}, volume~10 of {\em Philips Res. Reports Suppl.}
\newblock 1973.

\bibitem{G93}
C.~Godsil.
\newblock {\em Algebraic Combinatorics}.
\newblock Chapman \& Hall, New York, 1993.

\bibitem{G93a}
C.~Godsil.
\newblock Geometric distance-regular covers.
\newblock {\em New Zealand J. Math.}, 22:31--38, 1993.

\bibitem{H99}
A.~Hiraki.
\newblock Strongly closed subgraphs in a regular thick near polygon.
\newblock {\em European J. Combin.}, 20:789--796, 1999.

\bibitem{H09}
A.~Hiraki.
\newblock A characterization of some distance-regular graphs by strongly closed
  subgraphs.
\newblock {\em European J. Combin.}, 30:893--907, 2009.

\bibitem{KB10}
J.~Koolen and S.~Bang.
\newblock On distance-regular graphs with smallest eigenvalue at least $-m$.
\newblock {\em J. Combin. Theory Ser. B}, 100:573--584, 2010.

\bibitem{KQ1702}
J.~Koolen and Z.~Qiao.
\newblock Light tails and the \text{Hermitian} dual polar graphs.
\newblock {\em Des. Codes Cryptogr.}, 84:3--12, 2017.

\bibitem{L04}
M.~Lang.
\newblock A new inequality for bipartite distance-regular graphs.
\newblock {\em J. Combin. Theory Ser. B}, 90:55--91, 2004.

\bibitem{KQ17}
Z.~Qiao and J.~Koolen.
\newblock A valency bound for distance-regular graphs.
\newblock {\em Submitted to JCTA}.

\bibitem{DKT}
E.~van Dam, J.~Koolen, and H.~Tanaka.
\newblock Distance-regular graphs.
\newblock {\em Electron. J. Combin.}, (\#DS22), 2016.

\end{thebibliography}
\end{document}